\documentclass[11pt]{amsart}

\usepackage{amssymb,latexsym,pb-diagram}
\newtheorem{theorem}{Theorem}[section]
\newtheorem{prop}[theorem]{Proposition}
\newtheorem{lemma}[theorem]{Lemma}

\newtheorem{question}[theorem]{Question}
\newtheorem{definition}[theorem]{Definition}
\newtheorem{cor}[theorem]{Corollary}

\newtheorem{conj}[theorem]{Conjecture}

\def\R{{\bf\mathbb R} }

\begin{document}

\title{The Kodaira Dimension of Lefschetz Fibrations}
\author{Josef Dorfmeister}
\address{School  of Mathematics\\  University of Minnesota\\ Minneapolis, MN 55455}
\email{dorfmeis@math.umn.edu}
\author{Weiyi Zhang }
\address{School  of Mathematics\\  University of Minnesota\\ Minneapolis, MN 55455}
\email{zhang393@math.umn.edu}
\date{\today}

\begin{abstract}
In this note, we verify that the complex Kodaira dimension $\kappa^h$ equals the symplectic Kodaira dimension $\kappa^s$ for smooth $4-$manifolds with complex and symplectic structures. We also calculate the Kodaira dimension for many Lefschetz fibrations.
\end{abstract}

\maketitle

\tableofcontents

\section{Introduction}
Let $M$ be a smooth 4-manifold.  The manifold $M$ can be endowed with
different structures:  complex, symplectic, Lefschetz fibration,
etc.  If given a complex structure, the (complex) Kodaira
dimension $\kappa^h(M,J)$ is defined within the framework of
birational classification of complex manifolds. This notion has proven to be very useful,  at least in complex dimension $2$, in which we
have a detailed birational classification. Similarly, given a
symplectic structure, there is the concept of a (symplectic)
Kodaira dimension $\kappa^s(M,\omega)$ (see \cite{MS1},\cite{LB1} and \cite{L}). One would like to achieve a similar classification as in the complex case for symplectic manifolds.   Such a  classification
is very clear when the symplectic Kodaira dimension is $-\infty$.
Moreover, in symplectic Kodaira dimension $0$, T.-J. Li ( \cite{L} and \cite{L2} ) gives a
classification of symplectic $4$ manifolds up to their rational
homology groups (See also \cite{Sm} for a similar result). In Section \ref{lef0} (Prop. \ref{kod0}), we will give some evidence to
support his conjecture, that the classification in \cite{L2} is complete. 

The main goal of this paper is to show
the relationship between $\kappa^h(M,J)$ and $\kappa^s(M,\omega)$ on manifolds admitting complex and symplectic structures and
the role they play in the classification of Lefschetz
fibrations (and pencils).

In section 2, we state the definitions of the Kodaira
dimension in the complex and symplectic cases including some useful background information.  We then define the Kodaira dimension $\kappa^l(g,h,n)$ for Lefschetz fibrations with base genus $h\ge 1$, fiber genus $g$ and $n$ singular points (See Def \ref{lf}).

In Section 3, we show that the
symplectic Kodaira dimension $\kappa^s(M,\omega)$ as defined in \cite{L} is the same as
the complex Kodaira dimension for all smooth 4-manifolds admitting
complex and symplectic structures. That is, we have the following
theorem:

{\theorem \label{main1} Let $M$ be a smooth 4-manifold which
admits a symplectic structure $\omega$ as well as a complex structure $J$.
Then $\kappa^s(M,\omega)=\kappa^h(M,J)$.}

The symplectic structure $\omega$
 is not necessarily compatible with or even tamed by the
complex structure $J$. In the K\"ahler case, this result is known to T.-J.
Li (\cite{L}). This allows us to give the Kodaira dimension of a complex or
symplectic 4-manifold without ambiguity, we denote it simply as $\kappa(M)$.

In the following, we address the equivalence of $\kappa(M)$ and the Lefschetz Kodaira dimension $\kappa^l(g,h,n)$.  This is broken into three cases:  Excluding the exceptional case $(g,h,n)=(> 2,1,\ge 1)$, we first prove that on a complex manifold $M$
the Kodaira dimension $\kappa^l(g,h,n)$ coincides with the complex Kodaira dimension $\kappa^h(M,J)$:

{\theorem (See Theorem \ref{L=K}) \label{main2} If the complex manifold $(M,J)$ admits a smooth $(g,h,n)$ Lefschetz fibration with $h\ge 2$, then $\kappa^h(M,J)=\kappa^l(g,h,n)$.  This also holds when $h=1$, provided that either $g\le 2$ or $n=0$ (the case of a fiber bundle).

}

In Section \ref{sympeq}, we address the issue of equivalence on a symplectic Lefschetz fibration.  This differs very little from the complex case, the main result 
is

\begin{theorem}\label{main3}
Let $M$ be a symplectic manifold admitting the structure of a smooth $(g,h,n)$ Lefschetz fibration with $h\ge 2$, then $\kappa^s(M,\omega)=\kappa^l(g,h,n)$.  This also holds when $h=1$, provided that either $g\le 2$ or $n=0$ (the case of a fiber bundle).
\end{theorem}

Finally, we conjecture the equivalence of the dimensions in the exceptional case, providing some evidence for the conjecture.  This also has interesting connections to work by Amoros, et. al. \cite{AB}.

The results of Gompf and Matsumoto show that most Lefschetz fibrations admit a symplectic or complex structure. Therefore Lefschetz fibrations are a nice structure to research, especially as we hope that this will lead to a general definition
of Kodaira dimension for a much larger class of smooth
4-manifolds, and ultimately hopefully to a definition for broken Lefschetz fibrations (It has recently been shown that every smooth four manifold admits a broken Lefschetz fibration (\cite{AK}, \cite{Lek})). Additionally, this may give insight into a possible
extension of the symplectic Kodaira dimension to higher
dimensions.

It is worth noting at this juncture, that LeBrun has researched the connection between the Yamabe invariant of a smooth four manifold and Kodaira dimension on complex compact manifolds (see \cite{LB1}, \cite{LB2} and \cite{LB3}). This provides a possible generalization of Kodaira dimension to smooth 4-manifolds. We should mention that the equivalence of this ``Yamabe Kodaira dimension'' and symplectic Kodaira dimension is still not clear.

The flavor of Section \ref{lef0} is somewhat different from the previous
Sections. Here we attempt to give a combinatorial
definition of Kodaira dimension in the $h=0$ case, i.e. when the base is $\mathbb P^1$. Less is known in this (most
interesting, see Theorem \ref{symp}) case. However, this kind of combinatorial definition
may be generalized to more complicated fibration structures. 

The
authors hope to draw attention to the various
possible generalizations of Kodaira dimension in 4-dimensions.  They furthermore hope this will lead to a deeper understanding of the concept of Kodaira dimension and lead to useful definitions for smooth four-manifolds and manifolds of higher dimension.

{\bf Acknowledgments}  We would like to thank our thesis adviser
Professor T.-J. Li for many helpful comments and discussions.  We also thank M. Usher for his interest in our work.

We would like to express our gratitude to the referees, whose careful reading and very useful suggestions greatly improved the paper.

\section{The Kodaira Dimension}
In this section we will review the definitions of the Kodaira dimension for complex and symplectic 4-manifolds. Furthermore, we define the Kodaira dimension for Lefschetz fibrations.  We begin with the classic definition on a manifold admitting a complex structure and then consider manifolds admitting symplectic structures and Lefschetz fibrations with  positive genus  base.   Throughout, we include some relevant facts in preparation for the proofs of Theorems \ref{main1}, \ref{main2} and \ref{main3}.

\subsection{Complex Kodaira dimension}  If the manifold $M$ admits a complex structure $J$, then the
Kodaira dimension is defined as follows (in the case dim$_\R
M=4$, see Def \ref{2dim} for the 2 dimensional case):  The n-th plurigenus $P_n(M,J)$ of a complex manifold is defined
by $P_n(M,J)=h^0({\mathcal K}_J^{\otimes n})$, with ${\mathcal
K}_J$ the canonical bundle of $(M,J)$.  We denote by $c_1=c_1(X,J)$ the first Chern class of the complex manifold $(X,J)$. 

\begin{definition}The complex Kodaira dimension $\kappa^{h}(M,J)$ is defined as

\[
\kappa^{h}(M,J)=\left\{\begin{array}{cl} 
-\infty &\hbox{ if $P_n(M,J)=0$ for all $n\ge 1$},\\
0& \hbox{ if $P_n(M,J)\in \{0,1\}$, but $\not\equiv 0$ for all $n\ge 1$},\\
1& \hbox{ if $P_n(M,J)\sim cn$; $c>0$},\\
2& \hbox{ if $P_n(M,J)\sim cn^2$; $c>0$}.\\
\end{array}\right.
\]

\end{definition}

Let $(M,J)$ be a complex surface.  Then $M$ is called minimal if it does not contain a nonsingular rational curve of self-intersection $-1$.  A nonsingular surface $M_{min}$ is called a minimal model of $M$ if it is minimal and if $M$ can be obtained from $M_{min}$ by blowing up a finite collection of points.  It is known, that every compact nonsingular complex surface $M$ has a minimal model $M_{min}$ and, if $\kappa^h(M,J)\ge 0$, then the minimal model is unique up to isomorphism.  The classification of minimal compact complex surfaces is called the Enriques-Kodaira classification:

\begin{theorem}(\cite{K} or \cite{BPV}) \label{Kodaira}
Let $(M,J)$ be a minimal complex surface.  Denote by
$p_g$ the geometric genus of $M$ and $q$ the irregularity of $M$.  Then $(M,J)$ is classified according to the following table:  

\begin{tabular}{|c|c|c|}\hline
&Class&$\kappa^h$\\\hline
    1 &algebraic surfaces with $p_g=0$ & $-\infty, 0, 1, 2$\\\hline 
    2 &  K3 surfaces &$0 $\\\hline
    3 & complex tori (of dimension 2) &$ 0$\\\hline
    4 &  elliptic surfaces with $b_1$ even, $p_g
    \geqslant 1, c_1\neq 0$ & $0,1$\\\hline
   5 &  algebraic surfaces with $p_g \geqslant 1, c_1\geqslant
    0$, $c_1^2>0$ & $2$\\\hline
   6 &   elliptic surfaces with $b_1$ odd, $p_g \geqslant
    1$, $c_1^2=0$ & $0,1$\\\hline
    7 &  surfaces with $b_1=q=1, p_g=0$ & $-\infty$\\\hline
\end{tabular}
 
\end{theorem}

The first five classes  admit K\"ahler structures. The seventh
class of surfaces all have $b_2=0$ and hence admit no symplectic structures.  Details can be found in \cite{BPV}.

\subsection{Symplectic Kodaira dimension} Let $(M,\omega)$ be a symplectic manifold.  The symplectic canonical
class $K_\omega$ is the first Chern class of the cotangent bundle with any $\omega$-compatible almost complex structure. The set of $\omega$-compatible almost complex structures is nonempty and contractible, so $K_\omega$ is well-defined.  A symplectic manifold is called minimal if it does not contain a
symplectic embedded sphere with self intersection $-1$.  A symplectic manifold $(M,\omega)$ is symplectically minimal if and only if $M$ is smoothly minimal.

The
symplectic Kodaira dimension of a 2-manifold is defined as follows:

\begin{definition}\label{2dim}\[
\kappa^s(M,\omega)=\left\{\begin{array}{cc}
-\infty & \hbox{if $K_{\omega}<0$},\\
0& \hbox{ if $K_{\omega}=0$},\\
1& \hbox{ if $K_{\omega}> 0$}.
\end{array}\right.
\]
\end{definition}

Clearly, $\kappa(M,\omega)=-\infty$, $0$ or $1$ if and only if the
genus is 0, 1 or $\geqslant 2$.  These coincide with the complex Kodaira
dimension if dim$_\mathbb RM=2$.

A similar definition can be made for 4-manifolds, the symplectic Kodaira dimension $\kappa^s(M,\omega)$ is defined by Li \cite{L} (see also \cite{LB1}, \cite{MS}) to be:

{\definition \label{symp Kod}
For a minimal symplectic $4-$manifold $M$ with symplectic form $\omega$
and
symplectic canonical class $K_{\omega}$ the Kodaira dimension of $(M,\omega)$
 is defined in the following way:

\[
\kappa^s(M,\omega)=\left\{\begin{array}{cc}
-\infty & \hbox{if $K_{\omega}\cdot \omega<0$ or $K_{\omega}\cdot K_{\omega}<0$},\\
0& \hbox{ if $K_{\omega}\cdot \omega=0$ and $K_{\omega}\cdot K_{\omega}=0$},\\
1& \hbox{ if $K_{\omega}\cdot \omega> 0$ and $K_{\omega}\cdot K_{\omega}=0$},\\
2& \hbox{ if $K_{\omega}\cdot \omega>0$ and $K_{\omega}\cdot K_{\omega}>0$}.\\
\end{array}\right.
\]

The Kodaira dimension of a non-minimal manifold is defined to be
that of any of its minimal models.
}

 If the symplectic manifold carries a complex structure $J$ in addition to the symplectic structure, then we can define two classes:  The first Chern class $c_1(\mathcal K_J)$ of the canonical bundle $\mathcal K_J$ and the symplectic canonical class $K_\omega$. Please note that $c_1(\mathcal K_J)$ and $K_{\omega}$ may differ: If the manifold $M$ admits a K\"ahler structure and is minimal, then the first Chern class of the canonical bundle $\mathcal K_J$ is given by the canonical class $K_\omega$ of the K\"ahler form and it is unique up to diffeomorphism (see \cite{W} and \cite{FM}). If $J$ and $\omega$ are not compatible, then $c_1(\mathcal K_J)$ and $K_\omega$  are not necessarily equal. On a non-K\"ahler symplectic manifold there may be many symplectic canonical classes $K_\omega$ depending on the choice of symplectic structure $\omega$.  Hence the Kodaira dimension may depend on the choice of symplectic structure $\omega$.  The following Theorem addresses this issue:

\begin{theorem}\label{2.4}(Thm 2.4, \cite{L}) Let $M$ be a closed oriented smooth four manifold and $\omega$ an orientation compatible symplectic form on $M$.  Let $(M,\omega)$ be minimal. 
\begin{enumerate}
\item The Kodaira dimension is well defined.
\item $\kappa^s(M,\omega)$ only depends on the oriented diffeomorphism type of $M$, hence $\kappa^s(M,\omega)=\kappa^s(M)$.
\item $\kappa^s(M)=-\infty$ if and only if $M$ is rational or ruled.
\item $\kappa^s(M)=0$ if and only if $K_\omega$ is a torsion class.
\end{enumerate}
\end{theorem}

{\bf Remark:} The definition of the symplectic Kodaira dimension $\kappa^s(M,\omega)$ contains all possible combinations of $K_\omega\cdot K_\omega$ and $K_\omega\cdot \omega$ save one:  It was shown in \cite{L} that the pairing $K_\omega\cdot K_\omega>0$ and $K_\omega\cdot \omega=0$ is not possible.  The first statement of Thm. \ref{2.4} follows from this.

As in the complex category, certain results on minimal models hold:  In the symplectic category there exist minimal models, and for surfaces with $\kappa^s(M)\ge 0$ these are unique up to isomorphism.  Moreover, if $\kappa^s(M)=-\infty$, then there exist at most two minimal models up to diffeomorphism. 

To determine the Kodaira dimension of a symplectic manifold, it is necessary to determine the value of $K_{\omega}^2$.  Assume $M$ admits an almost complex structure.  Let $\chi$ denote the Euler characteristic and let $\sigma$ denote the signature of the manifold $M$.   Then 
\begin{equation}\label{K2}
K_\omega^2=2\chi+3\sigma.
\end{equation}  

Actually, what we really used in this paper for the definition of symplectic Kodaira dimension is a combination of Definition \ref{symp Kod} and Theorem \ref{2.4}.

\begin{cor}  \label{app}Let $M$ be a closed oriented smooth four manifold and $\omega$ an orientation compatible symplectic form on $M$.  Let $(M,\omega)$ be minimal.  $M$ is said to have symplectic Kodaira dimension
$\kappa^s(M)=-\infty$ if $M$ is rational or ruled. 
Otherwise, the Kodaira dimension $\kappa^s(M)$ of $M$ is defined in terms of $K_\omega$ as follows:
\[
\kappa^s(M)=\left\{\begin{array}{cc}
0& \hbox{ if $K_\omega$ is torsion },\\
1& \hbox{ if $K_\omega$ is non-torsion
but $K_\omega^2=0$},\\
2& \hbox{ if $K_\omega^2>0$}.\\
\end{array}\right.
\]
For a non-minimal $M$, $\kappa^s(M)$ is defined to be $\kappa^s(N)$, where
$N$ is a minimal model of $M$. 
\end{cor}

\subsection{The Kodaira dimension of Lefschetz fibrations}

The notion of a Lefschetz fibration is important for both symplectic and complex manifolds.  In particular, it can be viewed as a
topological characterization of symplectic manifolds. In the
following we present a definition of the Kodaira dimension for
Lefschetz fibrations when the base is not $S^2$. 

\subsubsection{Lefschetz Fibrations}
We begin with an overview of Lefschetz fibrations and their
connection to symplectic manifolds.

\begin{definition} \label{lf} A $(g,h)$ Lefschetz fibration on a
compact, connected, oriented smooth 4-manifold $M$ is a map $\pi:
M\rightarrow \Sigma_h$, where $\Sigma_h$ is a compact, connected,
oriented genus $h$ 2-manifold and $\pi^{-1}(\partial
\Sigma_h)=\partial M$, such that

\begin{itemize}
\item the set of critical points of $\pi$ is isolated and lies in
the interior of $M$; 
\item for any critical point $x$ there are
local complex coordinates $(z_1,z_2)$ compatible with the orientations on $M$ and $\Sigma_h$ such that
$\pi(z_1,z_2)=z_1^2+z_2^2$, 
\item $\pi$ is injective on the set of critical points and 
\item a regular fiber is a compact, connected,
oriented genus $g$ 2-manifold.
\end{itemize}
\end{definition}

Let $n$ denote
the number of singular points.  A singular fiber is a transversally immersed
surface with a single positive double point.  If there
are no critical points ($n=0$), then $\pi:M\rightarrow \Sigma_h$ is just a
surface bundle.    

A $(g,h)$ Lefschetz fibration is called relative minimal if no
fiber contains a sphere of self intersection $-1$.  

\begin{lemma}(\cite{S}) For $(g,h)$ Lefschetz fibrations with $h\ge 1$, relative minimal is
equivalent to $M$ minimal.
\end{lemma}

 The total space $M$ may admit symplectic structures $\omega$ or complex structures $J$, however these may not be compatible with the fibration structure.  This leads to the following definition.
 
\begin{definition}\begin{enumerate}
\item  A $(g,h)$ Lefschetz fibration is called symplectic if
there exists a symplectic form $\omega$ on $M$ such that, for any
$p\in \Sigma$, $\omega$ is nondegenerate at each smooth point of
the fiber $F_p$ and at each double point, $\omega$ is
nondegenerate on the two planes contained in the tangent cone.

\item A $(g,h)$ Lefschetz fibration is called holomorphic if there exists a complex structures $J$ and $j$ on $M$ resp. $\Sigma_h$ such that the map $\pi: (M,J)\rightarrow (\Sigma_h,j)$ is holomorphic.
\end{enumerate}
\end{definition}

The following two results show the
intimate connection between Lefschetz fibrations and symplectic
manifolds:

{\theorem\label{symp}\
 \begin{enumerate}

\item \cite{D} If $(M,\omega)$ is a symplectic 4-manifold, then
there exists a nonnegative integer $n$ such that the $n$-fold
blowup $M\#n\overline{\mathbb{CP}^2}$ admits a symplectic $(g,0)$ Lefschetz
fibration over $S^2$.
\item \cite{GS} Suppose $M$ is a 4-manifold
admitting a $(g,h)$ Lefschetz fibration $\pi:M\rightarrow
\Sigma_h$.  If the fiber class $F\in H_2(M)$ is nontrivial, then
$M$ admits a symplectic Lefschetz fibration structure.  In
particular, if $g\ne 1$ then this result holds.
\end{enumerate}}

Thus, given a $(g,h)$ Lefschetz fibration with nontrivial fiber class, we can
construct a symplectic structure on the underlying smooth manifold $M$.  This
in turn allows us to identify $M$ with a $(k,0)$ Lefschetz fibration
after a finite number of blow-ups.  The following question is interesting in particular from
the viewpoint of Kodaira dimension, which is invariant under
blow-up:

{\question \label{Sep} What is the connection between $(g,h,n)$
and $(k,0,n')$?  }

We cannot hope to obtain a precise answer to this question, as the number of singular fibers $n'$ resulting from Donaldson's construction is not bounded above.  In particular, the $(k,0,n')$ Lefschetz fibration structure is not unique.  However, the genus $k$ and the value $n'$ depend on each other.  We will consider this topic in Section \ref{lef0}.

We note two important properties of Lefschetz fibrations:  Recall the Definition of the Kodaira dimension for a 2-manifold given in Def. \ref{2dim}.

\begin{lemma}\label{subadd}
Let $M$ be the total space of a $(g,h)$ Lefschetz fibration.
\begin{enumerate}
\item(\cite{Ka})Assume the manifold $M$ admits a compatible complex structure $J$, i.e. a $J$ makes the Lefschetz fibration holomorphic.    Then
\begin{equation}
\kappa^h(M,J)\ge \kappa^h(\Sigma_g)+\kappa^h(\Sigma_h).
\end{equation}
\item(\cite{L3}) Assume the manifold $M$ admits a symplectic structure $\omega$.    Then
\begin{equation}
\kappa^s(M)\ge \kappa^s(\Sigma_g)+\kappa^s(\Sigma_h).
\end{equation}
\end{enumerate}
\end{lemma}

In order to determine the Kodaira dimension of a symplectic or complex manifold $M$, we need to estimate the square of the canonical class.  Assume $M$ admits a  relative minimal holomorphic Lefschetz fibration (allowing  $h=0$ as well).  Define the holomorphic Euler characteristic $\chi_h=\frac{\chi+\sigma}{4}$.  Then Xiao \cite{X} proved the slope inequality for fibrations with fiber genus $g\ge 2$: 
\begin{equation}\label{xiao}
K^2-8(g-1)(h-1)\ge (4-\frac{4}{g})(\chi_h-(h-1)(g-1)).
\end{equation}

Note that the assumption that the fibration is holomorphic is vital for this result to hold.  See also the discussion in Section \ref{sympeq}.

\subsubsection{Kodaira dimension}
We present a definition
of the Kodaira dimension for Lefschetz fibrations with $h\ge 1$.  The case $h=0$ will be discussed in Section \ref{lef0}.  This
definition is purely combinatorial and should thus be extend able
to more general structures on $M$.  In the following Sections, we proceed to show the
equivalence of this definition with $\kappa^s(M)$ or $\kappa^h(M,J)$ when $M$ admits the
structure of a Lefschetz fibration with $h\ge 1$ and a complex or
symplectic structure in all cases but the exceptional one.

\begin{definition} \label{Lef Kod}
Given a relative minimal $(g,h,n)$ Lefschetz fibration with
$h\ge 1$, define the Kodaira dimension $\kappa^l(g,h,n)$ as follows:

\[
\kappa^l(g,h,n)=\begin{cases}
-\infty &\hbox{if $g=0$ },\\
0& \hbox{if $(g,h,n)=(1,1,0)$  },\\
1& \hbox{if $(g,h)=(1, \ge 2)$ or $(g,h,n)=(1,1,>0)$ or $(\ge 2,1,0)$ },\\
2& \hbox{if $(g,h)\ge (2,2)$ or $(g,h,n)=(\ge 2,1,\ge 1)$}.
\end{cases}
\]

The Kodaira dimension of a non-minimal Lefschetz fibration with $h\ge 1$ is defined to be that of its minimal models.
\end{definition}

\noindent {\bf Remark:}  A priori, this definition does not necessarily provide a diffeomorphism invariant. Once Theorems \ref{main2} and \ref{main3} as well as Conjecture \ref{c1} have been proven, this will be an invariant up to diffeomorphism. Moreover, we will show later (Proposition \ref{diff}) that $\kappa^l$ is a diffeomorphism invariant even if we can prove only $\kappa^h=\kappa^l$.

\section{The Kodaira dimension on a complex or symplectic manifold}

We have defined three kinds of Kodaira dimension: the complex
$\kappa^h(M,J)$, symplectic  $\kappa^s(M)$ and the  Lefschetz $\kappa^l(g,h,n)$ Kodaira dimensions. In this Section, we prove the equivalence of these definitions when $M$ admits a complex structure. 

\subsection{$\kappa^s(M)=\kappa^h(M,J)$}

The K\"ahler case of the following Theorem was observed in \cite{L} and communicated to us by T.-J. Li.  We include his proof for the convenience of the reader.

{\theorem \label{nonK} Let $M$ be a smooth 4-manifold which admits
a symplectic structure $\omega$ as well as a complex structure $J$.  Then
$\kappa^s(M)=\kappa^h(M,J)$. }

\begin{proof}

Assume that $M$ is not minimal.  The Kodaira dimensions
$\kappa^s(M)$ and $\kappa^h(M,J)$ are invariant under blow-up and
blow-down.  Thus we may blow down any exceptional curves and work
only with the minimal manifolds thus obtained.  We may therefore
assume that $M$ is minimal in the following proof.

We consider first the K\"ahler case:  Let $(M, J, \omega)$ be a K\"ahler surface with
integrable complex structure $J$ and
K\"ahler form $\omega$. Let ${\mathcal K_J}$ be the canonical line bundle.
 Notice that the first Chern class of ${\mathcal K_J}$
is just $K_{\omega}$.  

Castelnuovo-Enriques proved  that K\"ahler surfaces
with $\kappa^h=-\infty$ are  rational or ruled surfaces. So if
$\kappa^h(M,J)=-\infty$ then $\kappa^s(M)=\kappa^h(M,J)$ by Theorem \ref{2.4}.

The plurigenera $P_n(M,J)$, hence the Kodaira dimensions,
 do not change under blowing down.
Moreover, for a K\"ahler surface with non-negative holomorphic Kodaira dimension,
 if the manifold is holomorphically minimal,
then it is smoothly minimal, in particular it is symplectically
minimal (see \cite{FM}). Notice this is not true when $(M, J)$ is a
rational surface, e.g. the Hirzebruch surface $F_3$ is
holomorphically minimal but not smoothly minimal). Thus, to
compare $\kappa^h(M, J)$ and $\kappa^s(M)$ when $\kappa^h(M,J)\geq
0$, we can assume that $(M,J)$ is holomorphically minimal.

If $(M, J)$ is a minimal
surface with $\kappa^h(M,J)=0$, then some positive power of
${\mathcal K_J}$ is a trivial holomorphic line bundle. In
particular, $c_1({\mathcal K_J})$ is a torsion class. Therefore
$\kappa(M)$ is equal to zero as well by Theorem \ref{2.4}.

If $(M, J)$ is a minimal surface with $\kappa^h(M,J)=1$, then it
is a minimal elliptic surface by the surface classification. In
particular, $c_1({\mathcal K}_J)$ has zero square. 

For a surface
$(M,J)$ with $\kappa^h(M,J)=2$, if $n\geq 2$, then
$$
P_n(M,J)=\frac{n(n-1)}{2}c^2_1({\mathcal K}_J)+\frac{1-b_1+b^+}{2}.
$$
In particular, $c_1({\mathcal K}_M)$ has positive square. 


As $K_\omega=c_1(\mathcal K_J)$ and applying Cor. \ref{app} prove $\kappa^s(M)=\kappa^h(M,J)$ if
$\kappa^h(M,J)$ is equal to 1 or 2, hence the result in the
K\"ahler case.

Assume now that $M$ admits a symplectic and a complex structure
but is not K\"ahler. Such manifolds exist, for example the Kodaira-Thurston manifolds.  Recall the classification given in Theorem \ref{Kodaira}.  It follows, that the only class of interest is class (6), as class (7) is not symplectic ($b_2=0$).

Therefore, consider elliptic surfaces in class $(6)$. Class $(6)$ surfaces with $\kappa^h(M,J)=0$ are just primary Kodaira surfaces.  It is shown in
\cite{L} that the symplectic Kodaira dimension of a primary Kodaira surface is 0.  A result of Biquard \cite{B} shows that for
$\kappa^h(M,J)=1$, symplectic is equivalent to K\"ahler.  Thus a non-K\"ahler class $(6)$ surface with $\kappa^h(M,J)=1$ doesn't admit a symplectic structure.  This finishes the proof in the non-K\"ahler case.

%
%
\end{proof}

\noindent {\bf Remark}\begin{enumerate}

\item The symplectic Kodaira dimension and complex Kodaira dimension in dimension $4$ are both diffeomorphism invariants.

\item For projective manifolds, the above theorem is related to
the abundance conjecture. See \cite{CP} for example. For the K\"ahler case, there are some
discussions in \cite{MS}.

\item For the above argument, we only need Biquard's result for properly elliptic surfaces with $p_g=1$.  If $p_g>1$, we can use the classification in \cite{L2}.  In this case, $b_1=2p_g+1\geqslant 5$. However, it was shown in \cite{L2}, that $b_1\le 4$ for all $M$ with $\kappa^s(M)=0$.  Hence $\kappa^s(M) = 1$ would follow without the results in \cite{B}.


\end{enumerate}

\subsection{$\kappa^l(g,h,n)=\kappa^h(M,J)$}\

This section is devoted to the proof of Theorem \ref{main2}.  We begin with some preliminary results which hold in both the complex and symplectic cases.

{\prop \label{20} Let $\pi:M\rightarrow \Sigma_h$ be a relative minimal $(g,h)$ Lefschetz fibration.  Assume that $M$ admits a complex or symplectic structure.  Denote the corresponding Kodaira dimension by $\kappa(M)$.

\begin{enumerate}
\item If $g\ge 2$ and $h\ge 2$, then $\kappa(M)=2$. 
\item If $g=0$, then $\kappa(M)=-\infty$.
\end{enumerate}}

\begin{proof}

Assume that $g,h\ge 2$, then  we obtain
$\kappa(\Sigma_g)=\kappa(\Sigma_h)=1$ by Def \ref{2dim}.  
The subadditivity of the  Kodaira dimension for Lefschetz fibrations (Lemma \ref{subadd}) states
\[ 
\kappa(M)\ge \kappa(\Sigma_g) + \kappa(\Sigma_h).
\]
Thus $\kappa(M)\ge 2$ if $g,h\ge 2$.

Assume $g=0$.  Then the Lefschetz fibration has no nontrivial
vanishing cycles, hence it has no singular points if our fibration
is relative minimal. We obtain a $S^2$ bundle over a Riemann
surface of genus $h$. This has $\kappa(M)=-\infty$ by Castelnuovo-Enriques (complex case) or Theorem \ref{2.4} (symplectic case).
\end{proof}

Applying the subadditivity of
the Kodaira dimension again, we obtain the following simple
corollary:

{\cor \label{sub} Assume that $M$ admits a complex or symplectic structure.  Denote the corresponding Kodaira dimension by $\kappa(M)$.
\begin{enumerate}
\item If $g=h=1$, then $\kappa(M)\geqslant 0$.
\item If $g\geqslant 2$ and $h \geqslant 1$ or $h\geqslant 2$ and $g
\geqslant 1$, then $\kappa(M)\geqslant 1$.
\end{enumerate}}

\begin{proof}
If $g\ge 2$, then we can apply Lemma \ref{subadd} due to Thm. \ref{symp}.  If $g=1$, then we can also apply Lemma \ref{subadd} when the total space is symplectic (in \cite{Wa} it was shown that $T^2$ bundles over a surface $\Sigma_g$ admit symplectic structures if and only if the fiber class is nontrivial). When the total space is complex, it has to be elliptic. Thus again we may apply Lemma \ref{subadd}.

\end{proof}

This corollary reduces the problem to the following two cases:

\begin{enumerate}
\item $g=1$ ;

\item $g\ge 2$ and $h=1$ (the exceptional case).
\end{enumerate}

We first consider the case $g=1$.  The following proposition will be useful in the discussion relating to this case, it holds in the more general framework of a surface bundle with either base or fiber a torus, but with the genus of the base positive:  In that case $\chi(M)=0$ by the
multiplicity of the Euler number.   A theorem of Meyer \cite{M} states that
$\sigma(M)=0$ if we have a surface bundle, such that the genus of the base is strictly positive and either the base or the fiber is a torus. The main point of the
proof is that we have a bound on the signature in terms of  $g$ and $h$,
but we can allow any $k$-fold multiple cover of $T^2$ by another
$T^2$. By pulling back by such a cover, we obtain a contradiction.
Thus we have:

{\prop \label{bdlT2} For a $(g,1)$ or $(1,g)$ surface bundle ($g\geqslant 2$) it follows that $\kappa(M)=1$. }

Once again, we let $\kappa$ denote one of $\kappa^s$ or $\kappa^h$, whichever is defined.

\begin{proof}
By Meyer \cite{M}, $\sigma(M)=0$, then
$K^2=3\sigma(M)+2\chi(M)=0$ (Eq. \ref{K2}). By Cor. \ref{sub}, $\kappa(M)\geqslant 1$ and  $M$ must have
$\kappa(M)=1$.
\end{proof}

We now return to the case where the fiber is a torus.  When the total space is complex, we have the following
famous result of Matsumoto (\cite{Mt}): Every relative minimal elliptic
Lefschetz fibration over a surface $\Sigma_h$ is either a torus
bundle over $\Sigma_h$ or $E(n,h)=E(n)\sharp_f \{\Sigma_h\times
T^2\}$ and all admit compatible complex structures. Thus we have:

{\prop \label{1h} If $g=1$, then the relative minimal $(1,h)$
Lefschetz fibrations with a complex or symplectic structure either have Kodaira dimension $-\infty$, $0$ or $1$. Specifically, letting $\kappa$ denote one of $\kappa^s$ or $\kappa^h$, whichever is defined,  we have:

\begin{enumerate}
\item $\kappa=-\infty$: $T^2 \times S^2$(Ruled surface), non-trivial $T^2$ bundle over $S^2$ (Hopf
surface), or $E(1)$ (non-minimal rational surface);

\item $\kappa=0$: $E(2)$, $T^2$ bundles over $T^2$;

\item $\kappa=1$: All the other cases.
\end{enumerate}
}

\begin{proof}
Consider first  torus bundles over $\Sigma_h$: 
\begin{enumerate}
\item $h=0$:  All such bundles
are complex and have $\kappa=-\infty$. More precisely, they are the following: The ruled surface
$T^2 \times S^2$ or the Hopf surface which is the non-trivial $T^2$ bundle over
$S^2$.

\item $h=1$: The bundles in this case are symplectic (\cite{G}) and have $\kappa=0$ (\cite{L}).

\item $h\geqslant 2$: Then $\kappa(M)=1$ due to the result of Prop. \ref{bdlT2}.  

\end{enumerate}

Consider now the case $E(n,h)=E(n)\sharp_f \{\Sigma_h\times
T^2\}$.  All of these manifolds are known to admit compatible complex structures.  We present a symplectic argument here:  When $h=0$, these are just the elliptic surfaces $E(n)$
which, by
 the classification of elliptic surfaces, have

\[\kappa = \begin{cases}
-\infty & E(1),\\
0& E(2),\\
1 & E(n)\ n> 2.\\
\end{cases}\]

When $h\geqslant1$, then $\kappa(\Sigma_h\times T^2)\geqslant 0$.
Furthermore,  we can check that the fiber sum $E(n,h)$ is
relative minimal. For example, given $E(1,h)$, then
$\Sigma_h\times T^2$ is already minimal, and in  the $E(1)$ part,
the fiber class is $3h-\sum_{i=1}^9e_i$, which  pairs positively
with any exceptional class $e_i$. For general $E(n,h)$, we can add
$E(1)$ step by step. Then it follows (Theorem 4.2, \cite{LY}), that $\kappa(E(n,h))=1$.

Now let the total space is symplectic and has at least one singular fiber, we know they are all holomorphic (see \cite{Moi}), and then it follows by the above discussion.

\end{proof}

\noindent {\bf Remark:}\begin{enumerate}
\item  Note that Prop. \ref{1h} gives information also for
the case in which the base is $S^2$, a case which we have excluded
in our Definition of the Kodaira dimension $\kappa(g,h,n)$. In
particular, this result as well as Prop \ref{20} include relatively
minimal Lefschetz fibrations with base $S^2$.

\item Part 2 of Prop. \ref{1h} can also be found in \cite{Sm}.  

\item In \cite{Wa} it was shown that $T^2$ bundles over a surface $\Sigma_g$ admit symplectic structures if and only if the fiber class is nontrivial.  When the fiber class is trivial, there exist families of complex surfaces (See class $(6)$, Thm \ref{Kodaira}).  However, we don't know if there are any $T^2$ fibrations over $\Sigma_g(g\ge 2)$, which are neither symplectic nor complex. 

\end{enumerate}

%
%
%

Thus we have proved the main result of this section:

{\theorem \label{L=K} Let $M$ admit a relative minimal $(g,h,n)$
Lefschetz fibration with $h\ge 1$ and a complex 
structure.  If $(g,h,n)\ne (> 2,1,\ge 1)$, then $\kappa^h(M,J)=\kappa^l(g,h,n)$.}

\vspace{0.2cm}

\noindent {\bf Remark.} \begin{enumerate}
\item The case $(g,h,n)= (2,1,>0)$ follows from Prop. \ref{he}.

\item  Assume that $M$ admits a $(g,1)$ holomorphic Lefschetz fibration.  Applying Xiao's result (Eq. \ref{xiao}) provides the estimate
\[
K^2\ge (4-\frac{4}{g})\chi_h
\]
which leads to the estimate
\begin{equation}\label{sigmaeuler}
\sigma\ge \frac{-(g+1)}{2g+1}\chi.
\end{equation}
This  estimate is particularly interesting when compared to the signature result in the hyperelliptic case, see Eq. \ref{Endo}.  It would be of great interest in the purely symplectic case as well (For example, our Conjecture \ref{c1} follows then), however it is an open question:  Question 5.10 posed in \cite{AB} in the symplectic setting would lead to this inequality if it is answered in the affirmative.

\end{enumerate}

\subsection{$\kappa^l(g,h,n)=\kappa^s(M)$}\label{sympeq}

\begin{theorem} 
Let $M$ be a smooth symplectic manifold admitting the structure of a $(g,h)$ Lefschetz fibration.  Then $\kappa^s(M)=\kappa^l(g,h,n)$ if $(g,h,n)\ne(>2,1,>0)$. 
\end{theorem}

\begin{proof}This follows from Prop. \ref{20}, Cor. \ref{sub}, Prop. \ref{1h} and Prop. \ref{bdlT2} as well as Prop. \ref{he} which addresses the $(g,h,n)=(2,1,>0)$ case. 

\end{proof}

\vspace{0.2cm}

\noindent {\bf Remark.} Notice that by Definition \ref{Lef Kod}, Theorem
\ref{main2} and Theorem \ref{main3} we actually have the following additivity relations $$\kappa^h(M)=\kappa^h(\Sigma_g) + \kappa^h(\Sigma_h)$$
$$\kappa^s(M)=\kappa^s(\Sigma_g) + \kappa^s(\Sigma_h)$$
for the Kodaira dimension of a Lefschetz fibration if the fibration is a bundle.  Additivity can be shown to hold for all the cases in Definition \ref{Lef Kod} by the notion of relative Kodaira dimension, see \cite{LZ}.

\section{The Exceptional Case: Conjectures and Results}

\begin{conj}
\label{c1} If $(g,h,n)=(>2,1,>0)$, then $\kappa^h(M,J)=2$ when $M$ is complex and $\kappa^s(M)=2$ when $M$ is symplectic.
\end{conj}

By
Cor. \ref{sub}, the only possible values are $\kappa^s(M)=1$ or
$\kappa^s(M)=2$. In other words, we only need to determine if $K_\omega^2=0$ or
not by the definition of the symplectic Kodaira dimension. We know
that $K_\omega^2=3\sigma(M)+2\chi(M)$.  Moreover, given a $(g,1)$ Lefschetz fibration, the Euler number is
determined by the number of vanishing cycles. In fact,
$$\chi(M)=\chi(\Sigma_g)\cdot \chi(T^2)+\sharp\{vanishing\
cycles\}=\sharp\{vanishing\ cycles\}.$$  Hence $\chi=n>0$.  Thus, if Eq. \ref{sigmaeuler} held, the result would follow.  

In the following two Propositions, we provide some evidence for the validity of this conjecture:

{\prop \label{e3}Given a $(g,1)$ Lefschetz fibration $(g\ge 2)$ with $n\ne 0\  mod\ 3.$  Then $\kappa^s(M) =2$.}

\begin{proof}
We know that $\kappa^s(M)\ge 1$ by Cor. \ref{sub}.  Assume that
$K_\omega^2=0$, i.e. that $\kappa^s(M)=1$.  Then

$$0=2\chi(M)+3\sigma(M)\;\;\Leftrightarrow\;\;-\frac{2}{3}\chi=\sigma(M)\in
\mathbb{Z}.$$

This is a contradiction, as $\chi(M)=n\ne 0$ mod $3$; thus
$\kappa^s(M)=2$.

\end{proof}

Using similar methods, we can show that the conjecture holds for all spin manifolds $M$ admitting a $(g,1)$ Lefschetz fibration if $24\nmid\chi(M)$.  This is essentially a result of Rokhlin's Theorem.  

If we assume that $M$ is complex, then the conjecture holds if $12\nmid\chi(M)$.  This assumption rules out the case $K_J^2=0$, as in the complex case $M$ would admit the structure of a properly elliptic fibration, for which 12 divides $\chi(M)$.

{\prop \label{he} Suppose the Lefschetz fibration is hyperelliptic,
i.e. the monodromy group of every singular fiber has a conjugation
to the hyperelliptic subgroup of the mapping class group. Then,
every hyperelliptic $(g,1)$ Lefschetz fibration $(g\geqslant 2)$
has $\kappa^s=2$. In particular, every $(2,1,>0)$ Lefschetz fibration
has $\kappa^s=2$. }

\begin{proof}
For every hyperelliptic $(g,1)$ Lefschetz fibration , we have a
signature formula in \cite{E}:
\begin{equation} \label{Endo}
\sigma(M)=-\frac{g+1}{2g+1}\mathfrak a+\sum_{h=1}^{[\frac{g}{2}]}
(\frac{4h(g-h)}{2g+1}-1)\mathfrak s_h.
\end{equation}

Here $\mathfrak a$ is the number of nonseparating vanishing cycles and $\mathfrak s_h$
denotes the number of separating vanishing cycles with one of the
separated components a genus $h$ surface. Denote
$\mathfrak s=\sum_{h=1}^{[\frac{g}{2}]}\mathfrak s_h$ to be the number of separating
vanishing cycles. Then $$\chi(M)=\sharp\{vanishing\
cycles\}=\mathfrak a+\mathfrak s.$$

Thus,

\begin{eqnarray*}
K^2&=&3\sigma(M)+2\chi(M)\\
 &=&\frac{g-1}{2g+1}\mathfrak a+\sum_{h=1}^{[\frac{g}{2}]}
\frac{6h(g-2h)+2g(h-1)+(4gh-1)}{2g+1}\mathfrak s_h\\
 &>&0,
\end{eqnarray*}
when $g\geqslant2$ and $\#\{\hbox{singular points}\}=\sharp\{\hbox{vanishing cycles}\}=a+s>0$.

Hence, every hyperelliptic $(g,1)$ Lefschetz fibration
$(g\geqslant 2)$ has $\kappa^s=2$. In particular, every $(2,1,>0)$
Lefschetz fibration is hyperelliptic, thus has $\kappa^s=2$.
\end{proof}

{\bf Remark:}  Consider the fiber sum $M=M_1\#M_2$ of two symplectic manifolds $M_1,M_2$ as defined by Gompf \cite{G1}.  Results in \cite{LY} provide an inequality for the Kodaira dimension $\kappa^s(M)$ with respect to $\kappa^s(M_i)$ and the surface $F$ along which the sum is performed:
\begin{equation}\label{LY}
\kappa^s(M)\ge \max\{\kappa^s(M_1), \kappa^s(M_2),\kappa^s(F)\}
\end{equation}

With regards to that result, we can produce more examples by taking the symplectic fiber sum of a hyperelliptic $(g,1)$ Lefschetz fibration with a $(g,0)$ Lefschetz fibration along the Lefschetz fiber and applying Eq. \ref{LY}.  This will produce a $(g,1)$ Lefschetz fibration with $\kappa^s(M)=2$ which generically should no longer be hyperelliptic.  The same could be done with a $(g,0)$ Lefschetz fibration with underlying manifold of Kodaira dimension $\kappa^s=2$.

At the end of this section, we want to show that, if the first part of the Conjecture holds, $\kappa^l$ is a diffeomorphism invariant. 

{\prop \label{diff} Assume that $\kappa^h(M,J)=2$ if $M$ is a complex surface and $(g,h,n)=(> 2,1,\ge 1)$. Then $\kappa^l$ is a diffeomorphism invariant.

}
\begin{proof} What we need to prove is that when $M$ admits two different Lefschetz fibrations with $h \ge 1$, these should have the same value for the Kodaira dimension $\kappa^l$.  The only case which needs consideration is the exceptional case of Conj \ref{c1}: $(g,h,n) = (> 2,1,\ge 1)$.  We will prove that, if a manifold $M$ admits a $(g,h,n)=(> 2,1,\ge 1)$ Lefschetz fibration,  then at the same time it does not admit a $(1,h,n)$ or $(\ge 2,1,0)$ fibration. We can rule out the case of a $(\ge 2,1,0)$ Lefschetz fibration because in both cases the last number in the triple is the Euler characteristic $\chi(M)$ of $M$.  Every $(1,h,n)$ Lefschetz fibration admits a complex structure by Proposition \ref{1h}; $M$ would be an elliptic complex surface.  However, we have assumed $\kappa^h(M,J)=2$ if $(g,h,n)=(\ge 2,1,\ge 1)$ (Conj. \ref{c1}) and hence this case can be ruled out.    
\end{proof}

Note, that if the first part of Conjecture \ref{c1} holds, then we have shown that a complex surface does not admit a $(\ge 2,1,\ge 1)$ and a $(1,h,n)$ Lefschetz fibration at the same time.

\section{\label{lef0}The Kodaira dimension of Lefschetz Pencils}

In the previous Section, we did not discuss in great detail the case
when the base is $S^2$. The main reason is that a relatively
minimal Lefschetz fibration is no longer necessarily minimal if the base
is $S^2$. For this reason it is interesting to consider Lefschetz
pencils. In this Section, we want to give some easy combinatorial
results in this most interesting case.  We begin by stating some results in \cite{D2}.

{\definition A Topological Lefschetz Pencil (TLP) on a compact
smooth oriented 4-manifold $X$ consists of the following data.

\begin{itemize}
\item Finite, disjoint subsets $A,B\subset X$, $A\ne \emptyset$; 
\item A smooth map
$f:X\backslash A \longrightarrow S^2$ which is a submersion outside $A\cup
B$; such that $f(b)\neq f(b')$ for distinct $b,b'\in B$ and which
is given in suitable oriented charts by the local models
$f(z_1,z_2)= z_2/z_1$ (in a punctured neighborhood of a point in
$A$) and $f(z_1,z_2)= z_1^2+z_2^2$ (in a neighborhood of a point
in $B$).
\end{itemize}
}

We shall generally also denote by $A$ (or $B$) the number of points in the respective sets, our definition implies $A>0$.

We can define two important classes: Firstly,  the
``hyperplane class" $\mathfrak h \in H^2(X;\mathbf{Z})$ which is the
Poincar$\acute{e}$ dual of the fiber class. Secondly, 
$K(X,f) \in H^2(X;\mathbf{Z})$ defined by
$$K(X,f)=-(c_1(V)+2f^*([S^2])).$$

Here $V$ is an oriented 2-plane bundle given by the tangent space
to the fibre of $f$. This can be extended to $X$ and the extension
is unique up to isomorphism. We should note that if $X$ carries a symplectic structure $\omega$ and the pencil is
compatible with  $\omega$, then
$K(\omega)=K(X,f)$.  The following topological facts can be found in \cite{D2}:

{\bf Facts:}

\begin{enumerate}
\item \label{1}The genus $k$ of a smooth fibre of $f$ is
$\frac{1}{2}(\mathfrak h\cdot \mathfrak h+K(X,f)\cdot \mathfrak h+2)$;

\item \label{2}$A=\mathfrak h \cdot \mathfrak h$;

\item \label{3} $B=\chi(X)+\mathfrak h \cdot\mathfrak  h+2K(X,f)\cdot\mathfrak  h$.

\end{enumerate}

We can now make precise the first statement of Theorem \ref{symp}:  If a smooth 4-manifold $X$ admits a TLP structure with hyperplane class $\mathfrak h$ such that $\mathfrak h \cdot\mathfrak  h>0$, then there exists a symplectic form in class $\mathfrak h$ compatible with the fibration structure.  Moreover, blowing up $X$ at $A$ points gives the manifold mentioned in Theorem \ref{symp}.  Conversely, a symplectic manifold $(X,\omega)$ admits a compatible TLP with hyperplane class $\tau[\omega]$ for $\tau$ sufficiently large.  This leads to the following definition:

{\definition \label{comb}  For any minimal smooth
4-manifold $M$ admitting a TLP $(k,A,B)$, we can define the
Kodaira dimension $\kappa^p(M)$ of the manifold $M$ as follows:

\[
\kappa^p(M)=\left\{\begin{array}{cc}
-\infty & \hbox{if $2k-2-A<0$ or $3\sigma(M)+2\chi(M)<0$},\\
0& \hbox{ if $2k-2-A=0$ and $3\sigma(M)+2\chi(M)=0$},\\
1& \hbox{ if $2k-2-A> 0$ and $3\sigma(M)+2\chi(M)=0$},\\
2& \hbox{ if $2k-2-A>0$ and $3\sigma(M)+2\chi(M)>0$}.\\
\end{array}\right.
\]

}

This combined
with the fact that the symplectic Kodaira dimension is a
diffeomorphism invariant, leads to  the following result:

\begin{theorem}
Let $(M,\omega)$ be a minimal symplectic manifold.  Then $\kappa^s(M)=\kappa^p(M)$.  Moreover, $\kappa^p(M)$ is well defined.
\end{theorem}

\begin{proof}
As stated above, we can find a TLP on $M$ such that the symplectic structure $\omega$ is compatible with the fibration.  Moreover, we can choose $[\omega]=\mathfrak h$ without loss of
generality.  As $\omega$ and the TLP are compatible, we have $K(X,f)=K_\omega$. Thus the definition coincides with our original one if
we know that $K(X,f)\cdot \omega=2k-2-A$. However, this is a quick
consequence of the  first fact above.

The definition is independent of the choice of the TLP because the
symplectic Kodaira dimension is an invariant of the diffeomorphism
type of the manifold and not dependent on the choice of symplectic structure $\omega$.
\end{proof}

\noindent {\bf Remark:}
\begin{enumerate}

\item The first half of the definition consists of combinatorial values of
the TLP, the second half of topological data. Notice that,
$\chi(M)$ also has a combinatorial expression in terms of $(k,A,B)$. For
hyperelliptic TLP, we have a purely combinatorial explanation of
$\sigma(M)$ by virtue of Endo's formula (\ref{Endo}). Thus, our definition is purely combinatorial for hyperelliptic TLP.

\item Given two TLP $(k,A,B)$ and $(k',A',B')$ for a minimal
4-manifold $M$ with $3\sigma(M)+2\chi(M)\geqslant0$.   Then
$2k-2-A>0$ if and only if $2k'-2-A'>0$.

\end{enumerate}

In the spirit of Question \ref{Sep}, we have the following result.

{\prop Given a  $(g,h,B')$ Lefschetz fibration and let $(k,0,A,B)$ be a corresponding
Lefschetz pencil.  Then $k=\frac{1}{4}(A+B-B')-gh+g+h$. }

\begin{proof}
It follows from Fact \ref{3} and Fact \ref{2} above that $K(X,f)\cdot\mathfrak  h=B-A-\chi(X)$.   Moreover, $\chi(X)=\chi(B)\cdot \chi(F)+B'$. Thus 
$$
K(X,f)\cdot
\mathfrak h=\frac{1}{2}(B-A-\chi(B)\cdot
\chi(F)-B')=\frac{1}{2}(B-A-B'-(2-2g)(2-2h)).
$$

 Finally,
\begin{eqnarray*}
k &=&\frac{1}{2}(\mathfrak h\cdot\mathfrak  h+K(X,f)\cdot\mathfrak  h+2)\\
 &=&\frac{1}{2}(A+\frac{1}{2}(B-A-B'-(2-2g)(2-2h))+2)\\
 &=&\frac{1}{4}(A+B-B')-gh+g+h
\end{eqnarray*}
\end{proof}

At the end of this Section, we want to say something about the
possible classification of symplectic manifolds with class
$\kappa=0$.  Li showed, that up to rational homology type, all symplectic manifolds with $\kappa^s=0$ are either the K3 surface, the torus $T^4$ or a $T^2$-bundle over $T^2$ (\cite{L} and \cite{L2}).  It was further conjectured, that this list is complete up to diffeomorphism.  The first result gives evidence of Li's conjectural
classification of symplectic manifolds with $\kappa=0$. The second
result gives some restrictions on such kind of manifolds.

{\prop \label{kod0} Let $M$ be a minimal symplectic 4-manifold with $\kappa(M)=0$.

\begin{itemize}

\item If $M$ admits a $(g,h)$ Lefschetz fibration with $h\geqslant
1$, then $M$ must be the $K3$ surface or a $T^2$-bundle over $T^2$. An
Enriques surface does not admit such a kind of Lefschetz fibration.

\item  For any 
TLP on $M$ the number $A$ must be even.  This is also true for the number of
singular fibers $B$.  Furthermore, the genus $k\ge 2$.
\end{itemize}
}

\begin{proof}
The first one is a Corollary of our classification. It's a
combination of Propositions \ref{20}, \ref{1h}, \ref{bdlT2} and
\ref{he}.

For the second one, notice that if $A=\mathfrak h \cdot\mathfrak  h>0$, then by the
discussion above Definition \ref{comb}, and Fact \ref{1}, we have $\mathfrak h
\cdot\mathfrak  h$ is even.  Similarly, $B=\chi(X)+\mathfrak h\cdot\mathfrak  h+2K(\omega)\cdot\mathfrak  h=\chi(X)+A$ and it
was shown in \cite{L2} that $\chi(X)\in \{0,12,24\}$.

The last result follows from a simple calculation using the above
Definitions and Prop. \ref{1h}.
\end{proof}

{\bf Remark:}\begin{enumerate}
\item The first part can be extended to the case of $S^2$ base and $S^2$ or $T^2$ fiber Lefschetz fibration by Proposition \ref{1h}. For the torus fiber case, see also Theorem $4.2$ in \cite{Sm}.
\item Consider the homotopy K3 surfaces constructed by Fintushel and Stern \cite{FS}.  These manifolds are constructed from knots $K$ in $S^3$.  0-framed knot surgery on $K$ produces a manifold $M_K$ containing a circle $m$ and hence the manifold $M_k\times S^1$ contains a torus $T$ with self-intersection 0.  Let $X$ be a K3 surface and $T_f$ a square 0 torus in $X$.  Then construct the symplectic fiber sum $X_K=X\#_{T_f=T}(M_K\times S^1)$.  It was shown in \cite{FS} that the Seiberg-Witten invariant of this manifold is given by the Alexander polynomial of the knot $K$.  Moreover, each of the $X_K$ is a homotopy K3 surface.  It admits a symplectic structure if the Alexander polynomial is monic. By the main theorem in \cite{LY}, $\kappa^s(X_K)=1$. 

\end{enumerate}

\end{document}